\documentclass[12pt,leqno]{amsart}
\usepackage{amssymb,amsmath,amsthm,amsfonts}
\usepackage{latexsym,enumerate}
\usepackage{mathrsfs,mathtools}
\usepackage{bbm}
\usepackage{xcolor}

\allowdisplaybreaks
\makeatletter
\def\subsection{\@startsection{subsection}{2}%
  \z@{.5\linespacing\@plus.7\linespacing}
{.5\baselineskip}%
  {\normalfont\centering\scshape}%
}
\makeatother

\newcommand{\nc}{\newcommand}
\nc{\les}{\lesssim}
\nc{\nit}{\noindent}
\nc{\nn}{\nonumber}
\nc{\D}{\partial}
\nc{\diff}[2]{\frac{d #1}{d #2}}
\nc{\diffn}[3]{\frac{d^{#3} #1}{d {#2}^{#3}}}
\nc{\pdiff}[2]{\frac{\partial #1}{\partial #2}}
\nc{\pdiffn}[3]{\frac{\partial^{#3} #1}{\partial{#2}^{#3}}}
\nc{\abs}[1] {\lvert #1 \rvert}
\nc{\cAc}{{\cal A}_c}
\nc{\cE}{{\cal E}}
\nc{\cF}{{\cal F}}
\nc{\cP}{{\cal P}}
\nc{\cV}{{\cal V}}
\nc{\cQ}{{\cal Q}}
\nc{\cGin}{{\cal G}_{\rm in}}
\nc{\cGout}{{\cal G}_{\rm out}}
\nc{\cO}{{\cal O}}
\nc{\Lav}{{\cal L}_{\rm av}}
\nc{\cL}{{\cal L}}
\nc{\cB}{{\cal B}}
\nc{\cZ}{{\cal Z}}
\nc{\cR}{{\cal R}}
\nc{\cT}{{\cal T}}
\nc{\cY}{{\cal Y}}
\nc{\cX}{{\cal X}}
\nc{\cXT}{{{\cal X}(T)}}
\nc{\cBT}{{{\cal B}(T)}}
\nc{\vD}{{\vec \mathcal{D}}}
\nc{\efield}{\mathcal{E}}
\nc{\vE}{{\vec \efield}}
\nc{\vB}{{\vec \mathcal{B}}}
\nc{\vH}{{\vec \mathcal{H}}}
\nc{\mR}{\mathcal R}
\nc{\mF}{\mathcal F}
\nc{\mE}{\mathcal E}
\nc{\ty}{{\tilde y}}
\nc{\tu}{{\tilde u}}
\nc{\tV}{{\tilde V}}
\nc{\Pc}{{\bf P_c}}
\nc{\bx}{{\bf x}}
\nc{\bX}{{\bf X}}
\nc{\bXYZ}{{\bf XYZ}}
\nc{\bY}{{\bf Y}}
\nc{\bF}{{\bf F}}
\nc{\bS}{{\bf S}}
\nc{\dV}{{\delta V}}
\nc{\dE}{{\delta E}}
\nc{\TT}{{\Theta}}
\nc{\dPsi}{{\delta\Psi}}
\nc{\order}{{\cal O}}
\nc{\Rout}{R_{\rm out}}
\nc{\eplus}{e_+}
\nc{\eminus}{e_-}
\nc{\epm}{e_\pm}
\nc{\sgn}{\text{sgn}}
\nc{\eps}{\varepsilon}
\nc{\vnabla}{{\vec\nabla}}
\nc{\G}{\Gamma}
\nc{\w}{\omega}
\nc{\mh}{h}
\nc{\mg}{g}
\nc{\vphi}{\varphi}
\nc{\tlambda}{\tilde\lambda}
\nc{\be}{\begin{equation}}
\nc{\ee}{\end{equation}}
\nc{\ba}{\begin{eqnarray}}
\nc{\ea}{\end{eqnarray}}

\nc{\g}{\gamma}
\nc{\ol}{\overline}

\newtheorem{theorem}{Theorem}[section]
\newtheorem{lemma}[theorem]{Lemma}
\newtheorem{prop}[theorem]{Proposition}
\newtheorem{corollary}[theorem]{Corollary}

\newtheorem{asmp}[theorem]{Assumption}

\nc{\pT}{\partial_T}
\nc{\pz}{\partial_z}
\nc{\pt}{\partial_t}
\nc{\la}{\langle}
\nc{\ra}{\rangle}
\nc{\infint}{\int_{-\infty}^{\infty}}
\nc{\halfwidth}{6.5cm}
\nc{\figwidth}{10cm}
\newcommand{\f}{\frac}

\nc{\nlayers}{L} \nc{\nsectors}{M}
\nc{\indicator}{\mathbf{1}}
\nc{\Rhole}{R_{\rm hole}}
\nc{\Rring}{R_{\rm ring}}
\nc{\neff}{n_{\rm eff}}
\nc{\Frem}{F_{\rm rem}}
\nc{\R}{\mathbb R}
\nc{\mJ}{\mathcal J}
\nc{\C}{\mathbb C}
\nc{\Z}{\mathbb Z}
\nc{\N}{\mathbb N}
\nc{\DD}{\Delta}
\nc{\cD}{\mathcal D}
\nc{\lnorm}{\left\|}
\nc{\rnorm}{\right\|}
\nc{\rnormp}{\right\|_{\ell^{p,\eps}}}
\nc{\rar}{\rightarrow} 

\newcommand{\norm}[1]{\left\lVert#1\right\rVert}

\begin{document}
	\DeclarePairedDelimiter{\gen}{\langle}{\rangle} 
	
 	\begin{abstract}
	 	We consider the higher order Schr\"odinger operator $H=(-\Delta)^m+V(x)$ in $n$ dimensions with real-valued potential $V$ when $n>4m$, $m\in \mathbb N$.    We adapt our recent results for $m>1$  to show that when $H$ has a threshold eigenvalue the wave operators are bounded on $L^p(\R^n)$ for the natural range $1\leq p<\frac{n}{2m}$ in both even and odd dimensions.  The approach used works without distinguishing even and odd cases, and matches the range of boundedness in the classical case when $m=1$.	The proof applies in the classical $m=1$ case as well and simplifies the argument.	 
	\end{abstract}
	
	\title[Wave operators for higher order  Schr\"odinger operators]{\textit{$L^p$-continuity of wave operators for higher order Schr\"odinger operators with threshold eigenvalues in high dimensions} } 
	
	\author[M.~B. Erdo\smash{\u{g}}an, W.~R. Green, K. LaMaster]{M. Burak Erdo\smash{\u{g}}an, William~R. Green, Kevin LaMaster}
	\thanks{ The first and third authors were partially supported by the NSF grant  DMS-2154031. The second author is partially supported by Simons Foundation
Grant 511825.  }
	\address{Department of Mathematics \\
		University of Illinois \\
		Urbana, IL 61801, U.S.A.}
	\email{berdogan@illinois.edu}
	\address{Department of Mathematics\\
		Rose-Hulman Institute of Technology \\
		Terre Haute, IN 47803, U.S.A.}
	\email{green@rose-hulman.edu}
	\address{Department of Mathematics \\
		University of Illinois \\
		Urbana, IL 61801, U.S.A.}
	\email{kevinl17@illinois.edu}

	\maketitle

\section{Introduction}

We continue the study of wave operators for higher order Schr\"odinger operators related to equations of the form
\begin{align*}
	i\psi_t =(-\Delta)^m\psi +V\psi, \qquad x\in \R^n, \quad  m\in \mathbb N.
\end{align*}
Here $V$ is a real-valued potential with polynomial decay, $|V(x)|\les \la x\ra^{-\beta}$ for some sufficiently large $\beta>0$, and some smoothness conditions, see \cite{EGWaveOp} or Assumption~\ref{asmp:Fourier} below.    Note that when $m=1$ this is the classical Schr\"odinger equation.  We consider the case  when $(-\Delta)^m+V$ has an eigenvalue at zero energy in high dimensions $n>4m$.  It is well known that there are no threshold obstructions other than eigenfunctions when $n>4m$.

Denote the free operator by $H_0=(-\Delta)^m$ and the perturbed operator by $H=(-\Delta)^m+V$.   We study the $L^p$ boundedness of the wave operators, which are defined by
$$
W_{\pm}=s\text{\ --}\lim_{t\to \pm \infty} e^{itH}e^{-itH_0}.
$$
For the class  of potentials   we consider, the wave operators exist and are asymptotically complete,  \cite{RS,ScheArb,agmon,Hor,Sche}.   Furthermore,  the intertwining identity
$$
f(H)P_{ac}(H)=W_\pm f((-\Delta)^m)W_{\pm}^*
$$
holds for these potentials where $P_{ac}(H)$ is the projection onto the absolutely continuous spectral subspace of $H$, and $f$ is any Borel function.  The intertwining identity and $L^p$ continuity of the wave operators allows one to obtain $L^p$-based mapping properties of operators of the form $f(H)P_{ac}(H)$ from those of the much simpler operators $f((-\Delta)^m)$.

As usual, we begin with the stationary representation of the wave operators
\begin{align}\label{eqn:stat rep}
	W_+u
	&=u-\frac{1}{2\pi i} \int_{0}^\infty \mR_V^+(\lambda) V [\mR_0^+(\lambda)-\mR_0^-(\lambda)] u \, d\lambda,
\end{align}
where $\mR_V(\lambda)=((-\Delta)^{m}+V-\lambda)^{-1}$, $\mR_0(\lambda)=((-\Delta)^m-\lambda)^{-1}$, and the `+' and `-' denote the usual limiting values as $\lambda$ approaches the positive real line from above and below, \cite{agmon,soffernew}.  As in previous works, \cite{EGWaveOp,EGWaveOpExt}, we consider $W_+$, bounds for $W_-$ follow by conjugation since $W_-=\mathcal C W_+ \mathcal C $, where $\mathcal Cu(x)=\overline{u}(x)$.  Since the identity operator is bounded on all $L^p$ spaces, it suffices to control the contribution of the integral involving the resolvent operators.

In dimensions $n>4m$, there are no resonances at the threshold (zero energy), \cite{soffernew}.  This mirrors the case of dimensions $n>4$ in the classical ($m=1$) Schr\"odinger operator where threshold resonances cannot exist. In the classical case, the existence of threshold eigenvalues limits the upper range of $L^p(\R^n)$ boundedness of the wave operators, generically to $1\leq p<\frac{n}{2}$, \cite{YajNew,GGwaveop,YajNew2}.  Here we prove the analogous result for the higher order Schr\"odinger operators in Theorem~\ref{thm:full} below.  Our proof doesn't distinguish between $m=1$ and $m>1$, and hence applies to the classical case as well where it simplifies the existing arguments.

Our main result is to control the low energy portion of the evolution when there is a threshold eigenvalue.  Our strategy builds on the approach in \cite{EGWaveOpExt}, extending the argument to control the singularities in the spectral parameter caused by the eigenvalue.
Using resolvent identity, one has
$$
\mR_V^+=\sum_{j=0}^{2k-1} (-1)^j \mR_0^+ (V\mR_0^+)^j +(\mR_0^+V)^{k} \mR_V^+ (V\mR_0^+)^k.
$$ 
The contribution of the $j$th term of the finite sum to \eqref{eqn:stat rep} is denoted by $W_j$ and the contribution of the remainder by $W_{r,k}$.  $W_j$ is unaffected by the existence of threshold obstructions, while $W_{r,k}$ is affected only when $\lambda$ is in a neighborhood of zero.   To make this more precise, take a smooth cut-off function $\chi\in C^\infty_0$  for a sufficiently small neighborhood of zero, and let $\widetilde \chi=1-\chi$ be the complementary cut-off away from a neighborhood of zero. Define
$$
W_{low,k}u= \frac{1}{2\pi i}  \int_{0}^\infty \chi(\lambda) (\mR_0^+(\lambda) V)^{k} \mR_V^+(\lambda)  (V\mR_0^+(\lambda) )^k V [\mR_0^+(\lambda)-\mR_0^-(\lambda)] u \, d\lambda,
$$
$$
W_{high,k}u= \frac{1}{2\pi i}  \int_{0}^\infty \widetilde\chi(\lambda) (\mR_0^+(\lambda) V)^{k} \mR_V^+(\lambda)  (V\mR_0^+(\lambda) )^k V [\mR_0^+(\lambda)-\mR_0^-(\lambda)] u \, d\lambda.
$$
Throughout the paper, we write $\la x\ra$ to denote $  (1+|x|^2)^{\f12}$, $A\les B$ to say that there exists a constant $C$ with $A\leq C B$, and write $a-:=a-\epsilon$ and $a+:=a+\epsilon$ for some $\epsilon>0$, and $\lceil c\rceil$ is the ceiling of $c$, the least integer greater than or equal to $c$.
Our main technical result is
\begin{theorem}\label{thm:main low}
	Let $n>4m\geq 4$. 
	Assume that $|V(x)|\les \la x\ra^{-\beta}$, where $V$ is a real-valued potential on $\R^n$ and $\beta>n+4$ when $n$ is odd and $\beta>n+3$ when $n$ is even.  If $H=(-\Delta)^m+V(x)$ has an eigenvalue at zero, but no positive
	eigenvalues, then $W_{low,k}$ extends to a bounded operator on $L^p(\R^n)$ for all $1\leq p<\frac{n}{2m}$ provided that $k$ is sufficiently large.
\end{theorem}
We need sufficiently large $k$ when $n> 4m$ due to local singularities of the free resolvents that are not  square integrable.  The main novelty is that the arguments are fairly streamlined, we avoid long operator-valued expansions of the perturbed resolvent by adapting the methods in \cite{EGWaveOp,EGWaveOpExt} to control the singularity as $\lambda \to 0$ that occurs when there is a zero energy eigenvalue.  We further note that no additional decay is needed on the potential compared to the regular case, \cite{EGWaveOp}.

To put this result in the context, recall  the first $L^p$ boundedness result in the seminal paper of Yajima, \cite{YajWkp1}, for $m=1$ and $1\leq p\leq \infty$ for small potentials.  For large potentials, the main difficulty is in controlling the contribution of $W_{low, k}$.  The behavior of this operator differs in even and odd dimensions. In \cite{YajWkp1,YajWkp2,YajWkp3}, Yajima removed smallness or positivity assumptions on the potential for all dimensions $n\geq 3$.  These arguments were simplified and Yajima further considered the effect of zero energy eigenvalues and/or resonances in \cite{Yaj} when $n$ is odd and with Finco in \cite{FY} when $n$ is even for $n>4$ to establish boundedness of the wave operators when $\frac{n}{n-2}<p<\frac{n}{2}$.  These results were further extended to show that the range of $p$ is generically $1\leq p<\frac{n}{2}$ in the presence of a zero energy eigenvalue, and that the upper range of $p$ may be larger under certain orthogonality conditions, \cite{YajNew,GGwaveop,YajNew2}.

We now give more details in the case $m>1$ to state the new corollary of our result above on the $L^p$ boundedness of wave operators.  
Let  $\mF({f})$ denote  the Fourier transform of $f$.
\begin{asmp}\label{asmp:Fourier}
	Fix $n>4m$ and $m\geq 1$.
	For some $0<\delta \ll 1$, $\sigma>\frac{2n-4m}{n-1-\delta}+\delta$, assume that the real-valued potential $V$ satisfies the condition
	$$\big\|\mathcal F(\la \cdot \ra^{\sigma} V(\cdot))\big\|_{L^{ \frac{n-1-\delta}{n-2m-\delta} }}<\infty.$$ 
	
\end{asmp}
In \cite{EGWaveOp}, by adapting Yajima's $m=1$ argument in \cite{YajWkp1}, the first two authors showed that the contribution of the terms of the Born series may be bounded by
\begin{align*}
	\|W_j\|_{L^p\to L^p}\leq C^j   \big\|\mathcal F(\la \cdot \ra^{\sigma} V(\cdot))\big\|_{L^{ \frac{n-1-\delta}{n-2m-\delta} }}^j, 
\end{align*}	
for some constant $C>0$.
In addition, it was shown that if $|V(x)|\les \la x\ra^{-\beta}$ for some $\beta>n+5$ when $n$ is odd and $\beta>n+4$ when $n$ is even and if $k$ is sufficiently large (depending on $m$ and $n$), then $W_{high,k}$ is a bounded operator on $L^p$ for all $1\leq p\leq \infty$ provided there are no positive eigenvalues.  The absence of positive eigenvalues is a common assumption for higher order operators since there may be positive eigenvalues even for smooth, compactly supported potentials, see \cite{soffernew}.

Combining these facts with Theorem~\ref{thm:main low}, we have the following   result.
\begin{corollary}\label{thm:full} 
	Fix  $m\geq 1$ and let $n>4m$. 
	Assume that  $V$  satisfies Assumption~\ref{asmp:Fourier} and in addition
	\begin{enumerate}[i)]
		\item $|V(x)|\les \la x\ra^{-\beta}$  for some $\beta>n+5$ when  $n$ is odd and 
		for some  $\beta>n+4$  when $n$ is even,
		\item $H=(-\Delta)^m+V(x)$ has an eigenvalue at zero energy, but no positive
		eigenvalues. 
	\end{enumerate}
	Then,	the wave operators extend to bounded operators on $L^p(\R^n)$ for all $1\leq p< \frac{n}{2m}$.  
	
\end{corollary}
By applying the intertwining identity and the known $L^p\to L^{p'}$ dispersive bounds when $p'$ is the H\"older conjugate of $p$ for the free solution operator $e^{-it(-\Delta)^m}$, we obtain the corollary below. 
\begin{corollary}
	Under the assumptions of Corollary~\ref{thm:full}, for $\frac{n}{n-2m}< p'\leq2$ we obtain the dispersive estimates
	$$
	\|e^{-itH}P_{ac}(H)f\|_{p} \les |t|^{-\frac{n}{ m}(\frac{1}{p'}-\frac{1}{2}) } \|f\|_{p'}.
	$$
	
\end{corollary}

The study of the $L^p$ boundedness of the wave operators in the higher order $m>1$ case has only recently begun.  In the case when there are no eigenvalues or resonances in the ac spectrum,   the case  $m=2$ and $n=3$ was studied by Goldberg and the second author, \cite{GG4wave}.  The case   $n>2m$ was studied by the first two authors in \cite{EGWaveOp,EGWaveOpExt}.  In \cite{EGG} the first two authors and Goldberg showed that a certain amount of smoothness of the potentials is necessary to control the large energy behavior in the $L^p$ boundedness. 

In \cite{MWY},  Mizutani, Wan, and Yao  considered the case of $m=2$ and $n=1$ showing that the wave operators are bounded when $1<p<\infty$, but not when $p=1,\infty$, where weaker estimates involving the Hardy space or BMO were proven depending on the type of threshold obstruction.  More recently in \cite{GY}, Galtbayar and Yajima considered the case   $m=2$ and $n=4$ showing that the wave operators are bounded on $1<p<\infty$ if zero is regular with restrictions on the upper range of $p$ if zero is not regular depending on the type of resonance at zero.  Mizutani, Wan, and Yao \cite{MWY3d,MWY3d2} studied the endpoint behavior and the effect of zero energy resonances when $m=2$ and $n=3$. This recent work on higher order, $m>1$, Schr\"odinger operators has roots in the work of Feng, Soffer, Wu and Yao \cite{soffernew} which considered time decay estimates between weighted $L^2$ spaces.

The paper is organized as follows.  In Section~\ref{sec:prop pf} we recall important resolvent expansions and  prove Proposition~\ref{lem:low tail low d}, which shows $L^p$ boundedness for a class of operators of the form needed in Theorem~\ref{thm:main low}.  In Section~\ref{sec:resolvents} we prove further resolvent expansions tailored to the case of zero energy eigenvalues to show that Proposition~\ref{lem:low tail low d} applies to the operator $W_{low,k}$ for large enough $k$, which suffices to prove Theorem~\ref{thm:main low} and consequently Theorem~\ref{thm:full}.  Finally in Section~\ref{sec:tech} we provide technical lemmas needed to prove Proposition~\ref{lem:low tail low d}.

\section{Operator Bounds}\label{sec:prop pf} 
In this section we reduce proving Theorem~\ref{thm:main low} to showing that a certain family of operators extend to bounded operators on $L^p(\R^n)$ on the desired range of $p$.  To do this we utilize various resolvent expansions and adapt the argument from the case when zero is regular to account for the extra  singularity in the spectral parameter that arises when $(-\Delta)^{2m}+V$ has a zero energy eigenvalue.

It is convenient to use a change of variables $\lambda\mapsto \lambda^{2m}$ to represent $W_{low,k}$ 
$$
\frac{m}{\pi i}\int_{0}^\infty \chi(\lambda) \lambda^{2m-1} (\mR_0^+(\lambda^{2m}) V)^{k } \mR_V^+(\lambda^{2m})  (V\mR_0^+(\lambda^{2m}) )^k V   [\mR_0^+(\lambda^{2m})-\mR_0^-(\lambda^{2m})]  \, d\lambda.
$$
We begin by using the symmetric resolvent identity on the perturbed resolvent $\mR_V^+(\lambda^{2m})$. With $v=|V|^{\f12}$, $U(x)=1$ if $V(x)\geq 0$ and $U(x)=-1$ if $V(x)<0$, we define $M^+(\lambda)=U+v\mR_0^+(\lambda^{2m})v$.   

Using the symmetric resolvent identity, one has
\be\label{eq:sym cons}
\mR_V^+(\lambda^{2m})V=\mR_0^+(\lambda^{2m})vM^+(\lambda)^{-1}v.
\ee
Therefore, we have
$$
W_{low,k}=	\frac{m}{\pi i}\int_{0}^\infty \chi(\lambda) \lambda^{-1} \mR_0^+(\lambda^{2m}) v\Gamma_k(\lambda) v   [\mR_0^+(\lambda^{2m})-\mR_0^-(\lambda^{2m})]  \, d\lambda,
$$
where $\Gamma_0(\lambda):= \lambda^{2m}M^+(\lambda)^{-1} $ and for $k\geq 1$
\begin{multline}\label{Gammalambda}
	\Gamma_k(\lambda)\\:=Uv\mR_0^+(\lambda^{2m}) \big(V\mR_0^+(\lambda^{2m})\big)^{k-1} v[\lambda^{2m}M^+(\lambda)^{-1}]v \big( \mR_0^+(\lambda^{2m})V\big)^{k-1} \mR_0^+(\lambda^{2m})vU.
\end{multline}
When zero energy is an eigenvalue, the resolvent $\mR_V$ becomes unbounded as $\lambda\to 0$.  Under the change of variables the singularity is of order $\lambda^{-2m}$.  The definition of $\Gamma_k(\lambda)$ above multiplies $M^{-1}(\lambda)$ by $\lambda^{2m}$ to account for this singularity.

To state our main result, we recall the following terminology from previous works involving wave operators and dispersive estimates.  An operator $T:L^2\to L^2$ with integral kernel $T(x,y)$ is absolutely bounded if the operator with kernel $|T(x,y)|$ is bounded as an operator on $L^2(\R^n)$.  We recall that finite rank operators and Hilbert-Schmidt operators are absolutely bounded, where $T$ is Hilbert-Schmidt if
$$
\|T\|_{HS}^2=\int_{\R^{2n}} |T(x,y)|^2\, dx\, dy <\infty.
$$

\begin{prop}\label{lem:low tail low d} 
	Fix  $n>4m\geq 4$  and let $\Gamma$ be a $\lambda $ dependent operator. Assume that 
	$$
	\widetilde \Gamma (x,y):= \sup_{0<\lambda <\lambda_0}\Big[\sup_{0\leq \ell\leq  \lceil \tfrac{n}2\rceil +1  } \big|\lambda^{\ell} \partial_\lambda^\ell \Gamma(\lambda)(x,y)\big| \Big]
	$$
	satisfies the pointwise estimate
	\be\label{eq:tildegamma}
	\widetilde \Gamma (x,y)  \les \la x\ra^{-\frac{n}2-}\la y\ra^{-\frac{n}2-}.
	\ee
	If $|V(x)|\les \la x\ra^{-\beta}$ for some $\beta>n_\star$
	where $n_{\star}=n+4$ if $n$ is odd and $n_{\star}=n+3$ if $n$ is even, then the operator with kernel 
	\be\label{Kdef}
	K(x,y)=\int_0^\infty \chi(\lambda) \lambda^{-1} \big[\mR_0^+(\lambda^{2m}) v \Gamma(\lambda)  v [\mR_0^+(\lambda^{2m}) -\mR_0^-(\lambda^{2m})]\big](x,y)   d\lambda 
	\ee
	is bounded on $L^p(\R^n)$ for $1\leq p<\frac{n}{2m}$.
\end{prop}
The claim of Theorem~\ref{thm:main low} follows in Section~\ref{sec:resolvents} by showing that $\Gamma_k(\lambda)$ defined in \eqref{Gammalambda} satisfies the hypotheses of this proposition.
To prove the proposition  we need the following representations of the free resolvent, these are Lemma 2.3, 2.4, and Remark 2.5 in \cite{EGWaveOpExt}, which have their roots in \cite{EGWaveOp}.
\begin{lemma}\label{prop:F} Let $n>2m\geq 2$. Then,   we have the representations  
	$$
	\mR_0^+(\lambda^{2m})(y,u)
	=  \frac{e^{i\lambda|y-u|}}{|y-u|^{n-2m}} F(\lambda |y-u| ).$$
	and
	$$
	[\mR_0^+(\lambda^{2m})-\mR_0^-(\lambda^{2m})](y,u)= \lambda^{n-2m}  \big[ e^{i\lambda |y-u|}F_+(\lambda |y-u|)+e^{-i\lambda |y-u|}F_-(  \lambda |y-u|)\big],
	$$
	With the bounds
	\be\label{Fbounds}
	|\partial_\lambda^N F(\lambda r)|\les \lambda^{-N} \la  \lambda r \ra^{\frac{n+1}2 -2m},\, \,\,\, \,\,\, 
	|\partial_\lambda^N F_\pm(\lambda r)|\les \lambda^{-N} \la  \lambda r\ra^{\frac{1-n}2}.
	\ee
	for all $N\geq 0$.
\end{lemma}
We say an operator $K$ with integral kernel $K(x,y)$ is admissible if
$$
\sup_{x\in \R^n} \int_{\R^n} |K(x,y)|\, dy+	\sup_{y\in \R^n} \int_{\R^n} |K(x,y)|\, dx<\infty.
$$
By the Schur test, it follows that an operator with admissible kernel is bounded on $L^p(\R^n)$ for all $1\leq p\leq \infty$. We are now ready to prove Proposition~\ref{lem:low tail low d}.

\begin{proof}[Proof of Proposition~\ref{lem:low tail low d}]
	
	Using the representations in Lemma~\ref{prop:F} with $r_1=|x-z_1|$ and $r_2:=|z_2-y|$ we see that $K(x,y)$ is the difference of  
	\begin{multline} \label{Kdefini}
		K_\pm(x,y)\\=\int_{\R^{2n}}  \frac{ v(z_1) v(z_2) }{r_1^{n-2m}   } \int_0^\infty e^{i\lambda (r_1 \pm r_2 )} \chi(\lambda) \lambda^{n-2m-1} \Gamma(\lambda)(z_1,z_2) F(\lambda r_1)F_\pm(\lambda r_2)  d\lambda dz_1 dz_2 .
	\end{multline}
	We write
	$$
	K(x,y) =: \sum_{j=1}^4 K_{j}(x,y),
	$$
	where the $K_j$ are $K$ restricted to different regions.  
	$K_1$ is the portion of $K$ restricted  to the set $r_1,r_2\les 1$, $K_2$ to the set $r_1 \approx r_2\gg 1 $,
	$K_3$ to the set $r_2\gg \la r_1\ra $, and $K_4$ to the set $r_1 \gg \la r_2\ra$. We define $K_{j,\pm}$ analogously. 
	
	Since both $\lambda$ and $r_j\les 1$ we have $\lambda r_j\les 1$. Using the bounds of Lemma~\ref{prop:F}  we bound the contribution of 	$|K_{1,\pm}(x,y)|$ by 
	\begin{multline*}
		\int_{\R^{2n}}  \frac{ v(z_1) v(z_2)\chi_{r_1,  r_2\les 1} }{r_1^{n-2m}   }   \widetilde\Gamma (z_1,z_2) \int^1_0\lambda^{n-2m-1}\ \!d\lambda dz_1 dz_2\\
		\lesssim\int_{\R^{2n}}  \frac{ v(z_1) v(z_2)\chi_{r_1,  r_2\les 1} }{r_1^{n-2m}   }   \widetilde\Gamma (z_1,z_2)dz_1 dz_2,
	\end{multline*}
	since $n-2m-1>-1$. 
	Using the pointwise decay  of $v$ and $\widetilde \Gamma$ in \eqref{eq:tildegamma}, we obtain
	$$\int  |K_{1,\pm}(x,y)|dy\les \int \la z_1\ra^{-n-}\la z_2 \ra^{-n-} r_1^{2m-n} dz_1dz_2 \les 1,  
	$$
	uniformly in $x$.  By the symmetry between $x$ and $y$
	which implies that $K_1$ is admissible.
	
	For $K_2$, we consider the contribution of $K_{2,-}$ since the $K_{2,+}$ is simpler. We integrate by parts twice in the $\lambda$ integral when $\lambda |r_1-r_2|\gtrsim1$ (using \eqref{Fbounds} and the definition of $\widetilde \Gamma$) and estimate directly when $\lambda |r_1-r_2|\ll1$ to obtain 
	\begin{multline*} 
		|K_{2,-}(x,y)|\les\\
		\int_{\R^{2n}} \frac{v(z_1)  \widetilde\Gamma (z_1,z_2) v(z_2)\chi_{r_1\approx r_2\gg 1 }}{  r_1^{n-2m}  } \int_0^\infty  \chi(\lambda) \lambda^{n-2m-1} \chi(\lambda|r_1-r_2|)  \la\lambda r_1\ra^{1-2m}    d\lambda dz_1 dz_2 \\
		+ \int_{\R^{2n}} \frac{v(z_1)  \widetilde\Gamma (z_1,z_2) v(z_2)\chi_{r_1\approx r_2\gg 1 }}{  r_1^{n-2m}  } \int_0^\infty  \frac {\chi(\lambda) \lambda^{n-2m-3} \widetilde\chi(\lambda|r_1-r_2|)   \la\lambda r_1\ra^{1-2m}}{|r_1-r_2|^2}     d\lambda dz_1 dz_2\\
		\les  \int_{\R^{2n}} \frac{v(z_1)  \widetilde\Gamma (z_1,z_2) v(z_2)\chi_{r_1\approx r_2\gg 1 }}{  r_1^{n-2m}  } \int_0^\infty  \frac {\chi(\lambda) \lambda^{n-2m-1}   \la\lambda r_1\ra^{1-2m}}{\la \lambda (r_1-r_2)\ra^2}     d\lambda dz_1 dz_2. 
	\end{multline*}
	Noting that $1-2m<0$, we can integrate this bound with respect to $x$ after converting to polar coordinates centered around $z_1$ to bound by
	\begin{multline*}
		\int|K_{2,-}(x,y)| dx \les\\
		\int_{\R^{2n}} \int_0^1 \int_{r_1\approx r_2\gg 1}  v(z_1)  \widetilde\Gamma (z_1,z_2) v(z_2)    r_1^{n-1}\frac {\lambda^{n-2m-1}(\lambda r_1)^{1-2m}}{r_1^{n-2m}\la \lambda (r_1-r_2)\ra^2}      d r_1 d\lambda dz_1 dz_2\\
		\les\int_{\R^{2n}} \int_0^1 \int_{r_1\approx r_2\gg 1}  v(z_1)  \widetilde\Gamma (z_1,z_2) v(z_2)    \frac {  \lambda^{n-4m}   }{\la \lambda (r_1-r_2)\ra^2}      d r_1 d\lambda dz_1 dz_2\\
		\les\int_{\R^{2n}} \int_0^1 \int_{\R}  v(z_1)  \widetilde\Gamma (z_1,z_2) v(z_2)    \frac {  \lambda^{n-4m-1}   }{\la  \eta \ra^2}      d \eta  d\lambda dz_1 dz_2\les 1,
	\end{multline*}
	uniformly in $y$. In the second line we defined $\eta=\lambda (r_1-r_2)$ in the $r_1$ integral and used   $n-4m-1\geq 0$.
	Since $r_1\approx r_2$, the integral in $y$ can be bounded uniformly in $x$ similarly and hence the contribution of $K_2$ is admissible.
	
	We now consider the contribution of
	\begin{multline} \label{K4ndef}
		K_{4,\pm}(x,y)= 
		\int_{\R^{2n}}  \frac{v(z_1) v(z_2)\chi_{r_1 \gg \la r_2\ra}}{r_1^{n-2m}}  \\  \int_0^\infty  e^{i  \lambda (r_1\pm r_2) } F(  \lambda r_1) \chi(\lambda)   \Gamma(\lambda)(z_1,z_2) \lambda^{n-2m-1}  F_\pm(\lambda r_2)  \ d\lambda dz_1 dz_2 .
	\end{multline}
	When $\lambda r_1\les 1$, using \eqref{Fbounds}, we bound $|F_\pm(\lambda r_2)|, |F(\lambda r_1)|\les 1$ and estimate the $\lambda$ integral  by 
	$r_1^{2m-n} \widetilde \Gamma(z_1,z_2)$, whose contribution to  $K_4$ is bounded by 
	$$
	\int_{\R^{2n}}  \frac{v(z_1) v(z_2) \widetilde \Gamma(z_1,z_2) \chi_{r_1 \gg \la r_2\ra}}{r_1^{n-2m+(n-2m)}} dz_1 dz_2 .
	$$
	Where apply Lemma~\ref{lem:admissible} with $\ell=n-4m$ and $k=n-2m$ to show that that $K_{4,\pm}$ are admissible kernels.
	
	When $\lambda r_1 \gtrsim 1$, we integrate by parts $N=\lceil\frac{n}{2}\rceil+1$ times 
	(using \eqref{Fbounds}) to obtain the bound
	\begin{align*}
		&\frac1{|r_1\pm r_2|^{N}}\int_0^\infty \Big|\partial_\lambda^{N} 
		\big[F(  \lambda r_1) \widetilde\chi(\lambda r_1)  \chi(\lambda)  \lambda^{n-2m-1} \Gamma(\lambda)(z_1,z_2)   F_\pm(\lambda r_2) \big] \Big| d\lambda\\ 	
		&\les r_1^{-N} \sum_{0\leq j_1+j_2+j_3+j_4\leq N, \, j_i\geq 0  }  
		\int_{\frac1{r_1}}^1    \lambda^{\frac{n+1}2-2m-j_1}  r_1^{\frac{n+1}2-2m}   \lambda^{n-2m-1-j_2}  \big|\partial_\lambda^{j_3}\Gamma(\lambda)(z_1,z_2)  \big| \frac{\lambda^{-j_4}}{\la \lambda r_2\ra^{\frac{n-1}2}} d\lambda\\
		&\les r_1^{ \frac{n+1}2-2m-N }  \widetilde \Gamma(z_1,z_2) \sum_{0\leq j_1+j_2+j_3+j_4\leq N, \, j_i\geq 0  }  
		\int_{\frac1{r_1}}^1    \lambda^{\frac{3n-1}2-4m-j_1-j_2-j_3-j_4}\gen{\lambda r_2}^{\frac{1-n}{2}} d\lambda\\
		&\les r_1^{ \frac{n+1}2-2m-N }  \widetilde \Gamma(z_1,z_2) \int_{\frac1{r_1}}^1\frac{\lambda^{\frac{3n-1}2-4m -N}}{\gen{\lambda r_2}^{\frac{n-1}{2}}}  d\lambda
	\end{align*}
	Since $\frac{n-1}{2}>0$ the contribution of $\gen{\lambda r_2}$ can be ignored.  Since $n> 4m$, $n-4m\geq 1$ and
	\[\frac{3n-1}{2}-4m-N\geq 1+\frac{n}{2}-\frac{1}{2}-\left\lceil\frac{n}{2}\right\rceil-1\geq-1.\]
	This gives us that the contribution from the $\lambda$ integral is bounded by $\log(r_1)\lesssim r_1^{\frac{1}{4}}$.  This gives us the total contribution to $K_4$ of
	\[\int_{\mathbb{R}^{2n}}\frac{v(z_1)v(z_2)\widetilde{\Gamma}(z_1,z_2)}{r_1^{\frac{n}{2}-\frac{3}{4}+N}} \,dz_1\, dz_2.\]
	Noting that $\lceil\frac{n}{2}\rceil+\frac{n}{2}+1-\frac{3}{4}\geq n+\frac{1}{4}$, again using that $r_1\gg r_2$, the contribution of this to \eqref{K4ndef} is bounded by
	$$
	\int_{\R^{2n}}  \frac{v(z_1) \widetilde{\Gamma}(z_1,z_2) v(z_2)\chi_{r_1 \gg \la r_2\ra}}{r_1^{n+\frac{1}{4}}} dz_1 dz_2,
	$$
	which is admissible  by Lemma~\ref{lem:admissible}.

	We now consider $K_3$ which restricts the upper portion of the range of $p$ to $1\leq p<\frac{n}{2m}$. Using Lemma~\ref{prop:F} we write
	\begin{multline}\label{K3}
		K_3(x,y):=
		\int_{\R^{2n}}  \frac{v(z_1) v(z_2)\chi_{r_2 \gg \la r_1\ra}}{r_1^{n-2m}} \,dz_1\, dz_2 \\
		\int_0^\infty  e^{i  \lambda (r_1\pm r_2) } F(  \lambda r_1) \chi(\lambda)  \lambda^{n-2m-1}   \Gamma(\lambda)(z_1,z_2)  F_\pm(\lambda r_2)  \ d\lambda  .
	\end{multline}
	When $\lambda r_2\lesssim 1$, we can apply \eqref{Fbounds} to obtain the bound
	\begin{align*}
		\int_{\R^{2n}}  \frac{v(z_1) v(z_2)\chi_{r_2 \gg \la r_1\ra}}{r_1^{n-2m}}&\widetilde{\Gamma}(z_1,z_2)\int^{r_2^{-1}}_0\gen{\lambda r_1}^{\frac{n+1}{2}-2m}\lambda^{n-2m-1}\gen{\lambda r_2}^{\frac{1-n}{2}} \ d\lambda dz_1 dz_2\\
		\lesssim&\int_{\R^{2n}}  \frac{v(z_1) v(z_2)\chi_{r_2 \gg \la r_1\ra}}{r_1^{n-2m}}\widetilde{\Gamma}(z_1,z_2)\int^{r_2^{-1}}_0\lambda^{n-2m-1}\ d\lambda dz_1 dz_2\\
		\lesssim&\int_{\R^{2n}}  \frac{v(z_1) v(z_2)\chi_{r_2 \gg \la r_1\ra}}{r_1^{n-2m}r_2^{n-2m}}\widetilde{\Gamma}(z_1,z_2) dz_1dz_2,
	\end{align*}
	which is bounded when $1\leq p<\frac{n}{2m}$ by Lemma~\ref{ lem:admissible3} with $k=\ell=0$.
	
	When $\lambda r_2 \gtrsim 1$, we integrate by parts $N=\lceil\frac{n}{2}\rceil +1$ times 
	(using \eqref{Fbounds}) to obtain the bound
	\begin{align}
		&\frac1{|r_1\pm r_2|^{N}}\int_0^\infty \Big|\partial_\lambda^{N} 
		\big[F(  \lambda r_1) \widetilde\chi(\lambda r_2)  \chi(\lambda)  \lambda^{n-2m-1} \Gamma(\lambda)(z_1,z_2)   F_\pm(\lambda r_2) \big] \Big| d\lambda\nonumber\\ 	
		&\les r_2^{-N} \sum_{0\leq j_1+j_2+j_3+j_4\leq N, \, j_i\geq 0  }  
		\int_{\frac1{r_2}}^1    \lambda^{-j_1}\gen{\lambda r_1}^{\frac{n+1}2-2m}\lambda^{n-2m-1-j_2}  \big|\partial_\lambda^{j_3}\Gamma(\lambda)(z_1,z_2)  \big|\frac{\lambda^{\frac{1-n}{2}-j_4}}{r_2^{\frac{n-1}2}} d\lambda\nonumber\\
		&\les r_2^{ \frac{1-n}2-N }  \widetilde \Gamma(z_1,z_2) \sum_{0\leq j_1+j_2+j_3+j_4\leq N, \, j_i\geq 0  }  
		\int_{\frac1{r_2}}^1    \lambda^{\frac{n-1}2-2m-j_1-j_2-j_3-j_4}\gen{\lambda r_1}^{\frac{n+1}{2}-2m} d\lambda\nonumber\\
		\label{K3MidStep}&\les r_2^{ \frac{1-n}2-N }    \widetilde \Gamma(z_1,z_2) \int_{\frac1{r_2}}^1\lambda^{\frac{n-1}2-2m -N}\gen{\lambda r_1}^{\frac{n+1}{2}-2m}d\lambda.\nn 
	\end{align}
	Here we note that derivatives of $\widetilde \chi(\lambda r_2)$ are comparable to division by $\lambda$.
	We consider the cases when $\lambda r_1\gtrsim 1$ and $\lambda r_1\lesssim 1$ as follows:
	\begin{multline*}
		r_2^{ \frac{1-n}2-N }    \widetilde \Gamma(z_1,z_2) \int_{\frac1{r_2}}^1\lambda^{\frac{n-1}2-2m -N}\gen{\lambda r_1}^{\frac{n+1}{2}-2m}d\lambda\\
		\les r_2^{-\frac{n-1}2-N} \widetilde\Gamma(z_1,z_2) 
		\Big(\int_{\frac1{r_2}}^{\min(\frac1{r_1},1)}  \lambda^{\frac{n }2-2m-N-\frac{1}{2}}\chi(\lambda r_1)d\lambda \\
		+  \int_{\min(\frac1{r_1},1)}^1 r_1^{\frac{n+1}2-2m}  \lambda^{n-4m-N}\widetilde\chi(\lambda r_1)d\lambda  \Big).
	\end{multline*}
	Since $\frac{n}{2}-2m-N-\frac{1}{2}=\frac{n}{2}-\lceil\frac{n}{2}\rceil-2m-\frac{3}{2}\leq-\frac{3}{2}$, the first integral is at most $r_2^{2m-\frac{n}{2}+N-\frac{1}{2}}$, so its contribution to \eqref{K3} is at most
	\be\label{eqn:K3Mid1}
	\int_{\R^{2n}}  \frac{v(z_1) v(z_2)\chi_{r_2 \gg \la r_1\ra}}{r_1^{n-2m}r_2^{n-2m}}\widetilde \Gamma(z_1,z_2)    dz_1 dz_2,
	\ee
	which is bounded for $1\leq p<\frac{n}{2m}$ by Lemma~\ref{ lem:admissible3}.  
	
	Similarly, after multiplying the second integral by $(\lambda r_1)^{N-1}\gtrsim 1$, 
	\[\int_{\min(\frac1{r_1},1)}^1 r_1^{\frac{n+1}2-2m}  \lambda^{n-4m-N}\widetilde\chi(\lambda r_1)d\lambda\lesssim\int_{\min(\frac1{r_1},1)}^1 r_1^{\frac{n+1}2-2m+N-1}  \lambda^{n-4m-1}d\lambda,\]
	since $n-4m-1\geq 0$, the $\lambda$ integral is bounded and the contribution is bounded by $r_1^{\frac{n+1}{2}-2m+N-1}$.  Letting $\{n/2\}=\lceil n/2\rceil-n/2$, the second integral's contribution to \eqref{K3} is at most
	\begin{multline*}
		\int_{\R^{2n}}  \frac{v(z_1) v(z_2)\chi_{r_2 \gg \la r_1\ra}r_1^{N-\frac{n}{2}-\frac{1}{2}}}{r_2^{N+\frac{n}{2}-\frac{1}{2}}}\widetilde \Gamma(z_1,z_2)    dz_1 dz_2\\
		=\int_{\R^{2n}}  \frac{v(z_1) v(z_2)\chi_{r_2 \gg \la r_1\ra}r_1^{\{n/2\}+\frac{1}{2}}}{r_2^{n+\{n/2\}+\frac{1}{2}}}\widetilde \Gamma(z_1,z_2)    dz_1 dz_2,
	\end{multline*}
	which is bounded in $L^p$ for all $1\leq p\leq \infty$ by Lemma~\ref{lem:admissible2}. 
\end{proof}

\section{Resolvent and Inverse Expansions}\label{sec:resolvents}
It remains only to prove that the operators $\Gamma_k(\lambda)$ defined in \eqref{Gammalambda} satisfy the bounds \eqref{eq:tildegamma} needed to apply Proposition~\ref{lem:low tail low d}.    In this section we develop different resolvent expansions and develop an expansion of $M^+(\lambda)^{-1}$ for small $\lambda$ when there is a threshold eigenvalue.  Throughout this section we consider the `+' limiting operators and omit the superscript.

Recall that $n_{\star}=n+4$ if $n$ is odd and $n_\star=n+3$ if $n$ is even. The bounds in Lemma~\ref{prop:F} imply that the operator $R_\ell$ with kernel
\be\label{R_k}
R_\ell(x,y):=v(x)v(y)\sup_{0<\lambda<\lambda_0}|\lambda^{\ell} \partial_\lambda^\ell \mR_0^+(\lambda^{2m})(x,y)|
\ee
satisfies 
$$
R_\ell(x,y) \les v(x) v(y) 
\big(|x-y|^{2m-n}+|x-y|^{\ell-(\frac{n-1}{2})}\big),\,\,\,\ell\geq 0.$$
This pointwise bound implies that  $R_\ell$ is bounded on $L^2(\R^n)$ for $0\leq \ell\leq \lceil \frac{n}2\rceil+1$ provided that $|V(x)|\les \la x \ra^{-\beta}$ for some $\beta>n_{\star}$, see \cite{EGWaveOp,EGWaveOpExt}. 

We write the iterated resolvent operators 
\begin{align}\label{eq:Alambda}
	A(\lambda, z_1,z_2) =  \big[ \big(\mR_0^+(\lambda^{2m})  V\big)^{k-1}\mR_0^+(\lambda^{2m})\big](z_1,z_2).
\end{align}  
For odd dimensions $n>4m$, if   $ k $ is sufficiently large depending on $n,m$ and $|V(x)|\les \la x \ra^{-n_\star-}$, then 
\begin{align*}
	\sup_{0<\lambda <1}| \lambda^{\ell} \partial_\lambda^\ell 	A(\lambda, z_1,z_2)|&\les \la z_1 \ra^2  \la z_2\ra^2,
\end{align*}
for $0\leq \ell\leq \frac{n+3}{2}=\lceil \frac{n}2\rceil+1$. This follows from the pointwise bounds on $R_\ell$ above. The iteration of the resolvents smooths out the local singularity $|x-\cdot|^{2m-n}$.  Each iteration improves the local singularity by $2m$, so that after $\ell$ iterations the local singularity is of size $|x-\cdot|^{2m\ell-n}$.  Selecting $k$ large enough ensures that the local singularity is completely integrated away. See the proofs of Propositions~5.3 and 6.5 in \cite{EGWaveOp} for more details. For even $n> 4m$, since we need fewer derivatives we have
$$	\sup_{0<\lambda <1}|\lambda^{\ell} \partial_\lambda^\ell 	A(\lambda, z_1,z_2)| \les \la z_1 \ra^{\frac32}  \la z_2\ra^{\frac32},
$$
for $0\leq \ell\leq \frac{n+2}{2}=\lceil \frac{n}2\rceil+1$.  More compactly, we have that
$$
\sup_{0<\lambda<1} |\lambda^\ell \partial_\lambda^\ell A(\lambda, z_1,z_2)|\les \la z_1\ra^{\{\frac{n}2\}+\frac32} \la z_2\ra^{\{\frac{n}2\}+\frac32}, \qquad 0\leq \ell \leq \lceil \frac{n}{2}\rceil +1.
$$

Finally, recalling that  
$$
\Gamma_k(\lambda)=Uv A(\lambda)v\lambda^{2m}M^{-1}(\lambda)v A(\lambda)vU,
$$
our goal is to show that $\sup_{0<\lambda<\lambda_0}|\lambda^{\ell+2m}\partial_\lambda^\ell[M(\lambda)]^{-1}(x,y)|$ is bounded on $L^2$ for each $0\leq\ell\leq\lceil\frac{n}{2}\rceil+1$.

We then define the following operators or their integral kernels
\begin{align*}
	G_0(x,y):=&a_0|x-y|^{2m-n}=\mR_0(0)(x,y),\\
	T_0:=&U+vG_0v.
\end{align*}
Here $a_0\neq 0$ is a real constant depending on $m$ and $n$.  Recall that the invertibility of $T_0$ on $L^2(\R^n)$ is equivalent to the absence of a threshold eigenvalue when $n>4m$, see \cite{soffernew}.
We further define $S_1$ to be the Riesz projection onto the kernel of $T_0$ so that the operator $T_0+S_1$ is invertible, and we define
\begin{align*}
	D_0:=&(T_0+S_1)^{-1}.
\end{align*}
From Proposition~2.4 in \cite{soffernew}, we have
\be\label{eqn:R_0 expansion}
\mR_0(\lambda^{2m})(x,y)=\sum^{N-1}_{j=0}\lambda^{2mj}G_j(x,y)+c_{n,m}\lambda^{n-2m} +E_0(\lambda)(x,y)
\ee
with $N=\lfloor\frac{n}{2m}\rfloor$. Here
$$
G_j(x,y)=c_{j,n,m}|x-y|^{2m-n+2mj}, j=0,...,N-1. 
$$
The exact value of these constants is unimportant for our purposes.  
We note that these expansions follow from those of the Schr\"odinger resolvents $R_0(z)=(-\Delta-z)^{-1}$ and the splitting identity
$$
\mR_0(z)(x,y):=((-\Delta)^m -z)^{-1}(x,y)=\frac{1}{ mz^{1-\frac1m} }
\sum_{\ell=0}^{m-1} \omega_\ell R_0 ( \omega_\ell z^{\frac1m})(x,y),
$$
where $\omega_\ell=\exp(i2\pi \ell/m)$ are the $m^{th}$ roots of unity.  The expansions utilize a significant amount of cancellation obtained from the splitting identity and the sum over the roots of unity.  See also the proofs of Lemma~4.2 and 6.2 in \cite{EGWaveOp}.  

We note that the error bounds for $E_0(\lambda)$ in \cite{soffernew} don't suffice for our purposes, instead we develop shorter expansions with more detailed control of the error term.
We make this more precise with the following lemmas.

\begin{lemma}\label{lem:R0 for Mexp}
	
	When $0<\lambda<\lambda_0$ and $0\leq \ell \leq \lceil\frac{n}{2}\rceil+1$, we have
	$$
	\mR_0(\lambda)=G_0+\lambda^{2m}G_1+E(\lambda)
	$$
	where for any $0<\epsilon<1$ we have
	$$
	\lambda^{\ell}|\partial_\lambda^{\ell} E(\lambda)(x,y)|\les \lambda^{2m+\epsilon}[|x-y|^{\{\frac{n}2\}+\frac32}+|x-y|^{4m-n+\epsilon}],
	$$
	where $\{n/2\}=\lceil n/2\rceil-n/2$.
	
\end{lemma}

\begin{proof}
	We need to consider cases as the resolvents behave differently in even and odd dimensions and based on the size of $\lambda|x-y|$.
	When $n$ is odd and $\lambda|x-y|\ll 1$, we have from  Lemma~2.3, equation (2.2), and Remark 2.2 in \cite{soffernew} that
	\begin{multline*}
		\mR_0(\lambda^{2m})(x,y)=\sum^{N-1}_{j=0}c_{j,n,m}\lambda^{2mj} |x-y|^{2m-n+2mj}
		+c_{n,m}\lambda^{n-2m} |x-y|^0\\
		+\sum_{j=0}^\infty c_{j,n,m}\lambda^{2mN+j}|x-y|^{2m-n+2mN+j}
	\end{multline*}
	with $N=\lfloor\frac{n}{2m}\rfloor$. Truncating after the first two terms of the series yields the expansion with 
	$$
	E(\lambda)= \sum^\infty_{k=1}c_k\lambda^{2m+k}|x-y|^{4m-n+k}, \,\,\,\lambda|x-y|\ll 1. 
	$$
	Note that many $c_k$ may be zero, this doesn't affect our bounds.
	This implies that for any $\ell$ we have
	$$
	\lambda^{\ell }\big|\partial^\ell_\lambda E(\lambda)(x,y) \chi(\lambda|x-y|)\big|\les  \lambda^{ 2m+1}|x-y|^{4m-n+1}, 
	$$
	which suffices for small $\lambda |x-y|$.
	
	When $n$ is even, there are logarithms when $\lambda|x-y|\ll 1$ that we must account for. Using Lemma~2.1 and Lemma~2.3 in \cite{soffernew}, (see also the expansion after Lemma 6.2 in \cite{EGWaveOp}) we have the following representation
	\begin{align*}
		\mR_0(\lambda^{2m})=&\sum^{\lceil\frac{n}{2m}\rceil-1}_{k=0}a_k\lambda^{2mk}|x-y|^{2m(k+1)-n}
		+\sum^\infty_{j=0}b_j\lambda^{n-2m+2j}|x-y|^{2j}\\
		&+\sum^\infty_{k=\lceil\frac{n}{2m}\rceil}c_k\lambda^{2m(k-1)}|x-y|^{2mk-n}\Bigl(\ln(|x-y|)+\ln\lambda \Bigr).
	\end{align*}
	However, since the first logarithm appears with a power of $\lambda$ at least as large as $n-2m$ since $\lceil\frac{n}{2m}\rceil \geq 3$, the logarithms don't affect the required bound and
	\begin{align*}
		E(\lambda)&=\sum^{\lceil\frac{n}{2m}\rceil-1}_{k=2}a_k\lambda^{2mk}|x-y|^{2m(k+1)-n}
		+\sum^\infty_{j=0}b_j\lambda^{n-2m+2j}|x-y|^{2j}\\
		&+\sum^\infty_{k=\lceil\frac{n}{2m}\rceil}c_k\lambda^{2m(k-1)}|x-y|^{2mk-n}\Bigl(\ln(|x-y|)+\ln\lambda \Bigr).
	\end{align*}
	The first two sums with only powers of $\lambda|x-y|$ are controlled as in the odd $n$ argument.  For the logarithms, since $k-1\geq 2$, the largest possible contribution is of the form
	$$
	\lambda^{4m} |x-y|^{6m-n}\ln(\lambda|x-y|)\les \lambda^{4m-1}|x-y|^{6m-n-1},
	$$
	using $\log(z)\les z^{-1}$ when $z\les 1$.  We may further divide by powers of $(\lambda |x-y|)$ to match the polynomial bound as needed. Since $m\geq 1$, we have
	$$
	\lambda^{\ell }\big|\partial^\ell_\lambda E(\lambda)(x,y) \chi(\lambda|x-y|)\big|\les  \lambda^{ 2m+1}|x-y|^{4m-n+1},
	$$
	which suffices for small $\lambda |x-y|$.
	
	The bound on the error term when $\lambda|x-y|\gtrsim 1$ follows from the bounds in Lemma~2.3 of \cite{EGWaveOpExt} and the definition of  the kernels of $G_0,G_1$. We have
	\begin{align*}
		\lambda^{\ell}|\partial_\lambda^\ell E(\lambda)|&=\lambda^{\ell}\partial_\lambda^\ell\big[\mR_0(\lambda^{2m})-G_0(x,y)-\lambda^{2m}G_1(x,y)\big]\\
		&\les   \lambda^{\frac{n+1}{2}-2m+\ell} |x-y|^{\frac{1-n}{2}+\ell}+|x-y|^{2m-n}+\lambda^{2m}|x-y|^{4m-n}  .
	\end{align*}
	We note that $\frac{1-n}{2}+\ell\leq \frac{1-n}{2}+\lceil \frac{n}2\rceil+1=\{\frac{n}2\}+\frac32$, and $\frac{n+1}{2}-2m+\lceil \frac{n}2\rceil+1>2m$.
	Since $\lambda |x-y|\gtrsim 1$ and $\lambda \les 1$, we can bound this by
	$$
	\lambda^{\ell}|\partial_\lambda^\ell E(\lambda)\widetilde\chi(\lambda|x-y|)|\les \lambda^{2m+\epsilon}[|x-y|^{\{\frac{n}{2}\}+\frac32}+|x-y|^{4m-n+\epsilon}].
	$$
\end{proof}

To invert $M(\lambda)$ in a neighborhood of the threshold, we utilize the Jensen-Nenciu inversion scheme in \cite{JN}, which requires some set up.
We introduce some notation to help streamline the upcoming statements and proofs.  For an absolutely bounded operator $T(\lambda)$ on $L^2(\R^n)$ we write $T(\lambda)=O_N(\lambda^j)$ to mean that
\be
\big\|\sup_{0<\lambda<\lambda_0}  \sup_{0\leq \ell \leq N} \lambda^{\ell-j}|\partial_\lambda^\ell T(\lambda)|\,\big\|_{L^2\to L^2}\les 1.
\ee
We note that, if $S,T$ satisfy $S(\lambda)=O_N(\lambda^{j})$ and $T(\lambda)=O_N(\lambda^{k})$, by the product rule the composition of the operators $ST$ is absolutely bounded and satisfies
$$
S(\lambda)T(\lambda)=O_N(\lambda^{j+k}) \qquad \text{on } (0,\lambda_0).
$$
Similarly,
$$
\lambda^j T(\lambda)=O_N(\lambda^{j+k}) \qquad \text{on } (0,\lambda_0).
$$
In particular, if $R$ is an absolutely bounded operator on $L^2$, then 
$$
T(\lambda)R,\, R T(\lambda)=O_N(\lambda^{k}) \qquad \text{on } (0,\lambda_0).
$$
Of particular use is the observation that if $T(\lambda)=O_{\lceil\frac{n}{2}\rceil+1}(\lambda^{0})$, then the operator with integral kernel
$$
\sup_{0<\lambda<\lambda_0} \sup_{0\leq \ell \leq \lceil\frac{n}{2}\rceil+1} \lambda^{\ell}|\partial_\lambda^\ell T(\lambda)(x,y)|
$$
is a bounded operator on $L^2(\R^n)$.

Finally, we have
\begin{lemma}\label{lem:opinv} If $\Gamma$ is a $ \lambda$-independent,  invertible, absolutely bounded operator on $L^2$ with an absolutely bounded inverse, then  for $\epsilon>0$, we have (for sufficiently small $\lambda_0$)
	$$
	[\Gamma+O_N(\lambda^\epsilon)]^{-1}=\Gamma^{-1}+O_N(\lambda^\epsilon)=O_N(\lambda^0).
	$$
\end{lemma}
\begin{proof}This is just a Neumann series expansion for $N=0$. For derivatives ($1\leq \ell\leq N$), note that  $\lambda^\ell\partial^\ell_\lambda\big[[\Gamma+O_N(\lambda^\epsilon)]^{-1}-\Gamma^{-1}\big]$ is a linear combination of operators of the form  
	$$
	[\Gamma+O_N(\lambda^\epsilon)]^{-1} \prod_{j=1}^J \big[\lambda^{\ell_j}\partial^{\ell_j}_\lambda O_N(\lambda^\epsilon) [\Gamma+O_N(\lambda^\epsilon)]^{-1}\big] =O_0(\lambda^\epsilon).
	$$ 
	Here $J\geq 1$, $1\leq \ell_j\leq N$ with $\sum \ell_j =\ell$.
\end{proof}

Our main result is the following proposition:
\begin{prop}\label{prop:Minv}
	
	If $\beta>n_\star$ and there is a zero energy eigenvalue, the operator $M(\lambda)$ is invertible on $L^2$ for sufficiently small $0<\lambda<\lambda_0$.  Furthermore, we have
	$$
	M^{-1}(\lambda)=O_{\lceil\frac{n}{2}\rceil+1}(\lambda^{-2m}).
	$$
	
\end{prop}
By the discussion above, these bounds on $M^{-1}(\lambda)$ will suffice to allow us to establish that the operators $\Gamma_k$ satisfy the hypotheses of Proposition~\ref{lem:low tail low d}.  To prove Proposition~\ref{prop:Minv}, we have the following series of lemmas.
\begin{lemma}\label{lem:M+S1}
	
	If $\beta>n_\star$, then for sufficiently small $\lambda$ the operator $M(\lambda)+S_1$ is invertible on $L^2$ with
	$$
	(M(\lambda)+S_1)^{-1}=O_{\lceil\frac{n}{2}\rceil+1}(\lambda^{0}).
	$$
	Further, for any $0<\epsilon< 1$, on $0<\lambda<\lambda_0$ we have
	$$
	(M+S_1)^{-1}=D_0-\lambda^{2m}D_0T_1D_0+O_{\lceil\frac{n}{2}\rceil+1}(\lambda^{2m+\epsilon}),
	$$
	where $		T_1=vG_1v.$

\end{lemma}

\begin{proof}
	
	For the first claim, recall that by Lemma~\ref{lem:R0 for Mexp}, we have
	\begin{align*}
		M(\lambda)=U+v\mR_0(\lambda^{2m})v=U+v(G_0+\lambda^{2m}G_1+E(\lambda))v,
	\end{align*}
	where for $0\leq \ell \leq \lceil\frac{n}2\rceil+1$
	$$
	\lambda^{\ell}|\partial_\lambda^\ell vE(\lambda)v|\les
	\lambda^{2m+\epsilon }v(x)|x-y|^{\{\frac{n}2\}+\frac32}v(y) +\lambda^{2m+\epsilon}v(x)|x-y|^{4m-n+\epsilon}v(y).
	$$
	The first term is Hilbert-Schmidt when $\beta>n_\star$ since
	$$
	[v(x)|x-y|^{\{\frac{n}2\}+\frac32}v(y)]^2 \leq \la x\ra^{-n-}\la y\ra^{-n-}.
	$$
	In dimensions $4m<n\leq 8m$, the second term  is Hilbert-Schmidt under weaker decay assumptions than $\beta>n_\star$.  On the other hand, if $n>8m$, the resulting kernel is too singular to be Hilbert-Schmidt, instead we identify $|x-y|^{4m-n+\epsilon}$ as a scalar multiple of the fractional integral operator $I_{4m+\epsilon}$.  By Lemma~2.3 in \cite{Jen}, $I_{4m+\epsilon}$ is bounded from $L^{2,s}\to L^{2,-s}$ provided  $s>2m+\frac{\epsilon}{2}$.  This suffices to show the $L^2$ boundedness of $v(x)\lambda^{\ell}|\partial_\lambda^\ell E(\lambda)|v(y)$ in dimensions $n>8m$ since $v(x)\les \la x\ra^{-\frac{\beta}2}$ with $\beta>n_*>4m+3$. We conclude that
	$$
	M(\lambda) =T_0 +\lambda^{2m}T_1+O_{\lceil\frac{n}{2}\rceil+1}(\lambda^{2m+\epsilon}). 
	$$
	
	We claim  that $D_0$ is an absolutely bounded operator.  To see this, first note that the argument of Lemma~4.3 in \cite{egt} may be adapted since, by the argument above, $vG_0v$ is absolutely bounded. Here one needs to iterate the resolvent identity sufficiently since $(vG_0v)^k$ is Hilbert-Schmidt for sufficiently large $k$ while $vG_0v$ itself is not.  We leave the details to the interested reader.
	
	Since $T_0+S_1$ is invertible on $L^2$ with an absolutely bounded inverse, by Lemma~\ref{lem:opinv},  we have 
	$$[M(\lambda)+S_1]^{-1}=O_{\lceil\frac{n}{2}\rceil+1}(\lambda^0).$$
	
	To obtain the second claim,   
	we utilize the resolvent identity $A^{-1}=B^{-1}+B^{-1}(B-A)A^{-1}$ with $A=(M(\lambda)+S_1)$ and $B=T_0+S_1$.  From this, we see that
	$$
	[M(\lambda)+S_1]^{-1}=D_0-D_0[M(\lambda)-T_0]D_0+(D_0[M(\lambda)-T_0])^2[M(\lambda)+S_1]^{-1}.
	$$
	Note that $M(\lambda)-T_0=\lambda^{2m}T_1+O_{\lceil\frac{n}{2}\rceil+1}(\lambda^{2m+\epsilon})=O_{\lceil\frac{n}{2}\rceil+1}(\lambda^{2m})$. Therefore,
	$$
	[M(\lambda)+S_1]^{-1}	=D_0-\lambda^{2m}D_0T_1D_0+O_{\lceil\frac{n}{2}\rceil+1}(\lambda^{2m+\epsilon}) 
	+O_{\lceil\frac{n}{2}\rceil+1}(\lambda^{4m}),
	$$
	which suffices. 
\end{proof}

To utilize the Jensen-Nenciu inversion machinery, Corollary~2.2 in \cite{JN}, we need to invert the operator
$$
\frac{1}{\lambda^{2m}}(S_1-S_1(M(\lambda)+S_1)^{-1}S_1)
$$
on $S_1L^2$.  To do so, we note that $S_1D_0=D_0S_1=S_1$ so that the leading $S_1$ cancels and consider the operator
\begin{multline}\label{eqn:B defn}
	B(\lambda)=\lambda^{-2m}[S_1-S_1(M(\lambda)+S_1)^{-1}S_1]= S_1T_1S_1-S_1\lambda^{-2m}O_{\lceil\frac{n}{2}\rceil+1}(\lambda^{2m+\epsilon})S_1 \\ =S_1T_1S_1+S_1 O_{\lceil\frac{n}{2}\rceil+1}(\lambda^{ \epsilon})S_1  .
\end{multline}
We now have 
\begin{lemma}\label{cor:Binv}
	
	If $\beta>n_\star$, for sufficiently small $\lambda$ the operator $B(\lambda)$ is invertible on $S_1L^2$.  Further, we have
	$$
	B^{-1}(\lambda)=O_{\lceil\frac{n}{2}\rceil+1}(\lambda^{0})
	$$
	as an operator on $S_1L^2$.
	
\end{lemma}

\begin{proof}
	
	Using the expansion in \eqref{eqn:B defn}, we note that $T_1$ is an invertible operator on $S_1L^2$ (c.f. Definition 2.6 and Remark 2.7 in \cite{soffernew}). Since $S_1L^2$ is finite dimensional, $T_1$ and its inverse are absolutely bounded.  The claim now follows from Lemma~\ref{lem:opinv}. 
\end{proof}

We are now ready to prove Proposition~\ref{prop:Minv}:
\begin{proof}
	
	By the Jensen-Nenciu inversion technique, \cite{JN}, we have
	\begin{align}
		M^{-1}(\lambda)=(M(\lambda)+S_1)^{-1}+\lambda^{-2m}(M(\lambda)+S_1)^{-1}S_1 B^{-1}(\lambda)S_1(M(\lambda)+S_1)^{-1},
	\end{align}
	provided that $B(\lambda)$ is invertible on $S_1L^2(\R^n)$. Therefore, the claim follows from 
	Lemmas~\ref{lem:M+S1} and \ref{cor:Binv}.
\end{proof}

We are now ready to prove the main technical result.

\begin{proof}[Proof of Theorem~\ref{thm:main low}]
	
	Proposition~\ref{prop:Minv} shows that $\sup_{0<\lambda<\lambda_0}|\lambda^{\ell+2m}\partial_\lambda^\ell[M(\lambda)]^{-1}(x,y)|$ is $L^2$ bounded for all $0\leq \ell\leq \lceil \frac{n}{2}\rceil+1$.   
	By \eqref{eq:sym cons}, the definition of $W_{low,k}$, $\Gamma_k(\lambda)$ and the discussion following \eqref{eq:Alambda}, we see that the operator $\Gamma_k(\lambda)$
	satisfies the hypotheses of Proposition~\ref{lem:low tail low d}.
	
\end{proof}

\section{Technical Lemmas}\label{sec:tech}
For completeness, we include necessary lemmas about the $L^p(\R^n)$ boundedness of certain integral kernels that are needed in the proof of Proposition~\ref{lem:low tail low d}.  We use the two Lemmas below about admissible kernels, which are Lemmas~4.1 and 4.2 in \cite{EGWaveOpExt} respectively.
\begin{lemma}\label{lem:admissible}
	Let $K$ be an operator with integral kernel $K(x,y)$ that satisfies the bound
	$$
	|K(x,y)|\les \int_{\R^{2n}} \frac{v(z_1)v(z_2) \widetilde \Gamma(z_1,z_2) \chi_{\{|y-z_2|\gg \la z_1-x\ra \} }}{|x-z_1|^{n-2m-k} |z_2-y|^{n+\ell} } \, dz_1 \, dz_2
	$$
	for some $0\leq k\leq n-2m$ and $\ell>0$.  Then, under the hypotheses of Lemma~\ref{lem:low tail low d}, the kernel of $K$ is admissible, and consequently $K$ is a bounded operator on $L^p(\R^n)$ for all $1\leq p\leq \infty$.
\end{lemma}

\begin{lemma}\label{lem:admissible2}
	
	Let $K$ be an operator with integral kernel $K(x,y)$ that satisfies the bound
	$$
	|K(x,y)|\les \int_{\R^{2n}} \frac{v(z_1)v(z_2) \widetilde \Gamma(z_1,z_2) \chi_{\{|y-z_2|\gg \la z_1-x\ra \}} |x-z_1|^{\ell} }{  |z_2-y|^{n+\ell} } \, dz_1 \, dz_2
	$$
	for some $\ell>0$.  Then, under the hypotheses of Proposition~\ref{lem:low tail low d}, the kernel of $K$ is admissible, and consequently $K$ is a bounded operator on $L^p(\R^n)$ for all $1\leq p\leq \infty$.
\end{lemma}

We also need the following new bounds for the analysis of $K_3$.
\begin{lemma}\label{ lem:admissible3}
	Let $K$ be an operator with integral kernel $K(x,y)$ that satisfies the bound
	\[|K(x,y)|\lesssim\int_{\mathbb{R}^n}\frac{v(z_1)v(z_2)\widetilde\Gamma (z_1,z_2)\chi_{\{|y-z_2|\gg\gen{z_1-x}\}}}{|x-z_1|^{n-2m-\ell}|y-z_2|^{n-2m+k}}\ \!dz_1dz_2\]
	for some $-2m<k<2m$ with $\ell<k+n-4m$.  Then, under the hypotheses of Proposition~\ref{lem:low tail low d}, $K$ is a bounded operator on $L^p(\mathbb{R}^n)$ for $1\leq p<\frac{n}{2m-k}$.
\end{lemma}

\begin{proof}
	We show that $\norm{K(x,y)}_{L^{p'}_yL^p_x}$ is bounded.
	Using the assumptions on $\widetilde{\Gamma}$ in Proposition~\ref{lem:low tail low d}, provided that  $\frac{n}{n-2m+k}<p'\leq\infty$ or equivalently $1\leq p<\frac{n}{2m-k}$, we may take the $L^{p'}_y$ norm to bound (recall that $r_1=|x-z_1|$ and $r_2=|z_2-y|$)
	\begin{align*}
		\norm{K(x,y)}_{L^{p'}_yL^p_x}\lesssim&\norm{\gen{z_1}^{-\frac{\beta+n}{2}}\gen{z_2}^{-\frac{\beta+n}{2}}r_1^{2m+\ell-n}\gen{r_1}^{2m-n-k+n/p'}}_{L^1_{z_1}L^1_{z_2}L^p_x}\\
		\lesssim&\norm{\gen{z_1}^{-\frac{\beta+n}{2}}\gen{z_2}^{-\frac{\beta+n}{2}}r_1^{2m-n+\ell}\gen{r_1}^{2m-k-n/p}}_{L^1_{z_1}L^1_{z_2}L^p_x}\stepcounter{equation}\tag{\theequation}\label{K3r_2small}.
	\end{align*}
	We then consider cases to control the upper bound in \eqref{K3r_2small}; when $r_1\lesssim 1$ and $r_1\gtrsim 1$.  If $r_1\gtrsim 1$, we can directly take the $L^p_x$ norm since $4m-n-(k-\ell)<0$ to see 
	\begin{align*}
		\Big\lVert\gen{z_1}^{-\frac{\beta+n}{2}}\gen{z_2}^{-\frac{\beta+n}{2}}r_1^{2m-n+\ell}&\gen{r_1}^{2m-k-n/p}\chi_{r_1\gg 1}\Big\rVert_{L^p_xL^1_{z_1}L^1_{z_2}}\\
		\lesssim&\Big\lVert\gen{z_1}^{-\frac{\beta+n}{2}}\gen{z_2}^{-\frac{\beta+n}{2}}r_1^{4m-n-n/p-(k-\ell)}\chi_{r_1\gg 1}\Big\rVert_{L^p_xL^1_{z_1}L^1_{z_2}}\\
		\lesssim&\norm{\gen{z_1}^{-\frac{\beta+n}{2}}\gen{z_2}^{-\frac{\beta+n}{2}}}_{L^1_{z_1}L^1_{z_2}} \les 1,
	\end{align*}
	since $\beta>n$.
	
	On the other hand, when $r_1\lesssim 1$, we leverage the decay of $v(z_1)$ and since all quantities are non-negative take the $L_{z_1}^1$ norm first. Converting to polar coordinates centered around $x$, and using that $\la z_1\ra \approx \la x\ra$, we have
	\begin{multline*}
		\Big\lVert\gen{z_1}^{-\frac{\beta+n}{2}}\gen{z_2}^{-\frac{\beta+n}{2}}r_1^{2m-n+\ell}\gen{r_1}^{2m-n-k+n-n/p}\chi_{r_1\lesssim 1}\Big\rVert_{L^1_{z_1}L^1_{z_2}L^p_x}\\
		\lesssim\norm{\gen{z_2}^{-\frac{\beta+n}{2}} \gen{x}^{-\frac{\beta+n}{2}} \int_{0}^1 r^{2m+\ell-1} \, \!dr}_{L^1_{z_2}L^p_{x}}\les
		\norm{\gen{z_2}^{-\frac{\beta+n}{2}}\gen{x}^{-\frac{\beta+n}{2}}}_{L^1_{z_2}L^p_{x}}\les 1,
	\end{multline*}
	since $2m+\ell-1>-1$ and $\beta>n$.
	
	Therefore the kernel of $K$ is a bounded operator on $L^p(\mathbb{R}^n)$ for $1\leq p<\frac{n}{2m-k}$.
\end{proof}

\section*{Funding}
Partial financial support was received from the Simons Foundation Grant 511825, and the National Science Foundation grant DMS-2154031.  The authors have no competing interests to declare that are relevant to the content of this article.

\end{document}